\documentclass{article}
\usepackage{amssymb}
\usepackage{amsthm}
\usepackage{systeme}
\usepackage{geometry}%
\usepackage[pagebackref]{hyperref}%
\usepackage{lmodern,mathtools}%
\usepackage[utf8]{inputenc}%
\usepackage[T1]{fontenc}%
\usepackage{todonotes}%
\usepackage{color}%
\usepackage{dsfont}
\usepackage{pgf,tikz,pgfplots}
\pgfplotsset{compat=1.14}
\usepackage{mathrsfs}
\usetikzlibrary{arrows}
\usepackage{enumitem}  

\newtheorem{theorem}{Theorem}[section]%
\newtheorem{proposition}[theorem]{Proposition}%
\newtheorem{remark}[theorem]{Remark}%

\title{Wishart processes : mean-field limit, long time behavior, and free probability}
\author{Ez\'echiel Kahn\footnote{CERMICS, Ecole des Ponts, INRIA, Marne-la-Vall\'ee, France. Email: ezechiel.kahn@enpc.fr}}
\date{\today}

\begin{document}

\maketitle

\begin{abstract}
This paper is devoted to the study of the eigenvalues of the Wishart process which are the analog of the Dyson Brownian Motion for covariance matrices. Such processes were in particular studied by Bru. The mean field convergence of the empirical measure of these eigenvalues was proved Malecki and Perez. In this paper, we provide a new approach to the mean field convergence problem using tools from the free rectangular convolution theory developed by Benaych-Georges, which in particular allows to compute explicitly the limit measure valued flow. We highlight the link with the integro-differential equation related to the mean field limit and its translation into a complex Burgers partial differential equation.
\end{abstract}


\section{Introduction }

We will denote by $\mathbb{N}=\{0,1,2,\dots\}$ the set of non-negative integers, and by $\mathbb{N}^*=\{1,2,3,\dots\}$ the set of positive integers. Let $n,m\in\mathbb{N}^*$ such that $n\leq m$ and let $(M_t)_t$ be a stochastic process taking its values in the space of  $n\times m$ matrices with real entries verifying the following stochastic differential equation (SDE)
\begin{equation}
\label{eq:matrix}
    dM_t = \kappa dW_t-\gamma M_tdt,
\end{equation}
where $W$ is a $n\times m$ matrix filled with independent Brownian motions, $\gamma\geq0$ and $\kappa\geq0$. $M$ is thus a matrix whose entries are independent Ornstein-Uhlenbeck processes just as the one considered in \cite{MR1132135} by Bru. The reader will find in \cite{MR1871699} an analysis of the complex analog of Bru's model.

For all $t\ge0$ let   $(x^{1,n}_t,\dots,x^{n,n}_t)$ and $(\lambda^{1,n}_t,\dots,\lambda^{n,n}_t)$ be respectively the eigenvalues of the $n\times n$ square matrices $\sqrt{\frac{1}{m}M_tM_t^*}$ and $\frac{1}{m}M_tM_t^*$ where $*$ is the transposition operator. 

It is proved in \cite{MR991060} and \cite{MR1132135} that the eigenvalues $(\lambda^{1,n}_t,\dots,\lambda^{n,n}_t)$  satisfy the system of SDEs
\begin{equation}
\label{eds}
    d\lambda^{i,n}_t=\frac{2\kappa }{\sqrt{m}}\sqrt{\lambda^{i,n}_t}dB^i_t+\kappa^2dt-2\gamma\lambda^{i,n}_tdt+\frac{\kappa^2 }{m}\sum_{j\neq i}\frac{\lambda^{i,n}_t+\lambda^{j,n}_t}{\lambda^{i,n}_t-\lambda^{j,n}_t}dt \text{ for all }i\in\{1,\dots n\}
\end{equation}
\[0\leq\lambda^{1,n}_t<\dots<\lambda^{n,n}_t\text{ a.s. dt-a.e.},\]
where $B^1,\dots, B^n$ are independent Brownian motions.  It is proved for instance in \cite{MR3296535} that the SDE (\ref{eds}) admits a strong pathwise unique solution.

 If $(\lambda^{1,n}_t,\dots,\lambda^{n,n}_t)_t$ is solution to the SDE (\ref{eds}), then a formal calculus gives for $x^{i,n}=\sqrt{\lambda^{i,n}}$: 
\begin{equation*}
dx_t^{i,n} =\frac{\kappa }{\sqrt{m}}dB_t^{i}+ \left(\left(1-\frac{1 }{m}\right)\frac{\kappa^2}{2x_t^{i,n}}-\gamma x_t^{i,n} +\frac{\kappa^2 }{2mx_t^{i,n}}\sum_{j\ne i}\frac{(x_t^{i,n})^2+(x_t^{j,n})^2}{(x_t^{i,n})^2-(x_t^{j,n})^2}\right)dt 
\end{equation*}
\[0\leq x^{1,n}_t<\dots<x^{n,n}_t\text{ a.s. dt-a.e.}\]

We respectively define the empirical measure of the eigenvalues of $\frac{1}{m}M_tM_t^*$ and the symmetrized empirical measure of the eigenvalues of   $\sqrt{\frac{1}{m}M_tM^*_t}$  by
\begin{equation*}
    \nu^n_t =\frac{1}{n}\sum_{i=1}^n\delta_{\lambda^{i,n}_t} \quad \text{and}\quad  \mu^n_t = \sqrt{\nu^n_t}= \frac{1}{2n}\sum_{i=1}^n\left(\delta_{x^{i,n}_t}+\delta_{-x^{i,n}_t}\right),
\end{equation*}
where here and in the rest of the paper, for any probability measure $\mu$ on $[0,+\infty)$, $\sqrt{\mu} = \mathrm{sym}(\sqrt{.}\sharp\mu)$ denotes the symmetrization of the push-forward of $\mu$ by the map $v\mapsto\sqrt{v}$. The application $\mathrm{sym}$ is fully defined in Section \ref{sect:preliminaries}. These measure sequences are random variables sequences in the space of probability measures on respectively $\mathbb{R}_+$ and $\mathbb{R}$. We will thus speak about weak convergence in probability or almost surely throughout the paper in the sense that the sequence converges in probability or almost surely in the  space of probability measures on $\mathbb{R}_+$ (or $\mathbb{R}$) with the distance of Prokhorov (see for instance \cite{billingsley2013convergence}).
Let us note that by  non negativeness of the $(\lambda^{1,n}_t,\dots,\lambda^{n,n}_t)_{ t\geq0}$, symmetry of $\mu^n_t$ and continuity of the functions $v\mapsto v^2$ and $v\mapsto \sqrt{v}$, there is equivalence between the convergence of $(\nu^n)_n$ and the convergence of $(\mu^n)_n$.

In this note, we are interested in the convergence and limit of these measure processes when $n$ and $m$ go to infinity at a rate $n/m\rightarrow\alpha\in(0,1]$.

Our main result gives a complete answer to the limit of $(\mu^n)_n$  problem, using free probability tools (in particular the rectangular free convolution of index $\alpha$ denoted by $\boxplus_\alpha$ defined in Section \ref{sect:preliminaries}).  
Let us first denote by $\mu^{\mathrm{MP}_{\rho, \sigma}}$  the Marcenko-Pastur distribution with shape parameter $\rho$ and scale parameter $\sigma$ which admits the density with respect to the Lebesgue measure : \[x\mapsto\frac{\sqrt{( a_+-x)(x-  a_-)}}{2\pi\rho x \sigma^2}\mathds{1}_{[ a_-, a_+]} \text{ where }a_\pm=\sigma^2(1\pm\sqrt{\rho})^2.\]

\begin{theorem}[Mean-field limit]\label{th:mfl}
Let us assume $n\leq m$, and that when $n$ grows to infinity, $m$ grows to infinity too, with
$\frac{n}{m}\rightarrow\alpha\in(0,1]$.
  Let us also assume  that $(\mu^n_0)_n$ converges weakly in probability towards a  non random probability measure $\mu_0$. Then for all $t\geq0$,
  ${(\mu^n_t)}_{n\geq0}$ converges weakly in probability and
  \[
    \mu_t:=\lim_{n\to\infty}\mu^n_t=(e^{-\gamma t}\mu_0)\boxplus_\alpha \sqrt{\mu^{\mathrm{MP}_{\alpha, \sigma_t}}},
  \]
  where $\boxplus_\alpha$ is the rectangular free convolution of  parameter $\alpha$, $e^{-\gamma t}\mu_0$ denotes the  push-forward of $\mu_0$ by the map $v\mapsto e^{-\gamma t}v$,  and where for all $t\geq0$ :
  \begin{equation*}
      \sigma_t^2 =\left\{
      \begin{aligned}
      &\frac{\kappa^2}{2\gamma}(1-e^{-2\gamma t}) \quad\text{if}\quad \gamma\neq0\\
      &\kappa^2t \quad\text{if}\quad \gamma=0.
      \end{aligned}
      \right.
  \end{equation*}
 Moreover, if $\gamma\neq0$ and with $\sigma_\infty = \sqrt{\frac{\kappa^2}{2\gamma}}$, we have :
 \[\underset{t\rightarrow+\infty}{\lim}\mu_t = \sqrt{\mu^{\mathrm{MP}_{\alpha,\sigma_\infty  }}}.\]
\end{theorem}
For $\gamma=0$ and $\kappa=1$, this measure valued flow $(\mu_t)_t$ can be seen as the law of the free Wishart process introduced in \cite{capitaine2005free}. The reader will find in \cite{biane1997free,biane1998stochastic,biane2001free} an introduction to the free stochastic calculus theory.

This last result is in fact true in a complex framework  (with complex Brownian motions of the form $(B^1_t+\mathrm{i}B^2_t)_{t\geq0}$ where $(B^1_t)_{t\geq0}$ and $(B^2_t)_{t\geq0}$ are independent real Brownian motions), as the reader will see in Section \ref{sect:complex}. This result might also be extended to a quaternionic framework but the rectangular free convolution theory is not developed yet in this case. We moreover proved the following result.
 
 \begin{theorem}[Commutativity of the limits]\label{th:limcommut}
Under the assumption of Theorem \ref{th:mfl}, let us moreover suppose $\gamma\neq0$ and that for all $n\in\mathbb{N}^*$: $\sum_{i=1}^n\mathbb{E}[\lambda^{i,n}_0]<+\infty$. Let us note $\sigma_\infty=\sqrt{\frac{\kappa^2}{2\gamma}}$. Then we have in the sense of weak convergence 

\[\underset{\underset{n/m\rightarrow\alpha}{n\rightarrow\infty}}{\lim} \underset{t\rightarrow\infty}{\lim}\nu^{n}_t = \underset{t\rightarrow\infty}{\lim} \underset{\underset{n/m\rightarrow\alpha}{n\rightarrow\infty}}{\lim} \nu^{n}_t = \mu^{MP_{\alpha,\sigma_\infty}}.\]

\end{theorem}
 
 \paragraph{}
The SDE (\ref{eds}) is related to the $\beta$-\emph{Wishart} (or  $\beta$-\emph{Laguerre}) process, the system of particles defined by the SDE
\begin{equation*}
    d\lambda^{i,n}_t=2 \sqrt{\lambda^{i,n}_t}dB^i_t-2\eta\lambda^{i,n}_tdt+ \delta dt +\beta\sum_{j\neq i}\frac{\lambda^{i,n}_t+\lambda^{j,n}_t}{\lambda^{i,n}_t-\lambda^{j,n}_t}dt \text{ for all }i\in\{1,\dots n\}
\end{equation*}
\[0\leq\lambda^{1,n}_t<\dots<\lambda^{n,n}_t\text{ a.s. dt-a.e.}\]
with $\beta,\delta,\eta>0$ and which, for $\beta=1$, $\eta=m\gamma/\kappa^2$, $\delta=m$, and after the change of time $t\rightarrow \frac{\kappa^2}{m}t$, corresponds to the dynamics of the process followed by the eigenvalues of $\frac{1}{m}M_tM_t^*$. Sometimes in the literature, the $\beta$-\emph{Wishart} process is considered with $\eta=0$ (see for instance \cite{Demni}) or with $\eta=1/2$ (see for instance \cite{trinh2020betalaguerre}). Sometimes, the $\beta$-Wishart process is also found in the literature as the system of particles given by the SDE
\begin{equation*}
    d\lambda^{i,n}_t=2 \sqrt{\lambda^{i,n}_t}dB^i_t-2\eta\lambda^{i,n}_tdt+ \beta\left(\delta +\sum_{j\neq i}\frac{\lambda^{i,n}_t+\lambda^{j,n}_t}{\lambda^{i,n}_t-\lambda^{j,n}_t}\right)dt \text{ for all }i\in\{1,\dots n\}
\end{equation*}
\[0\leq\lambda^{1,n}_t<\dots<\lambda^{n,n}_t\text{ a.s. dt-a.e.,}\]
as for instance in \cite{Allez}.

The reader will find in  Section \ref{sect:generalwishart} an extension of Theorems \ref{th:mfl} and \ref{th:limcommut} to the $\beta$-Wishart case.

\paragraph{}
Similar problems where tackled before, often by the direct study of the limit of empirical measure processes of a system of interacting diffusive particles. In \cite{MR1217451}, the authors studied the empirical measure of diffusing particles with electrostatic repulsion and a negative linear elastic force. They proved that the accumulation points  of the sequence of empirical measures of the particles sequence  (the sequence being indexed by the number of particles)  satisfy an integro-differential equation. The method of their proof was used in other papers and is reproduced here in our context. They proved the uniqueness of the limit by translating the integro-differential equation problem into a complex partial differential equation (PDE) problem, a Burgers equation, which they could solve using the method of characteristics.

In \cite{MR1440140} was studied a more general system of diffusing particles with electrostatic repulsion, a linear elastic force, and an additional constant drift. In this case, an integro-differential equation for the limit of the empirical measure of the particles was found and translated into a Burgers-like complex PDE using the  methods developed in \cite{MR1217451}. However, the existence and uniqueness of the solution to the PDE were proved using Fourier transforms methods.  A second article from the same authors \cite{MR1875671} followed, studying Brownian particles with electrostatic repulsion on the circle. There, the same procedure allowed the authors to get  another Burgers-like PDE of order two,  the existence and uniqueness of a solution to which were this time proved using the Hopf-Cole transformation. This allowed the authors to find an explicit expression of the limit measure flow when each particle starts from the position $0$, i.e. when the initial condition of the empirical measure is $\delta_0$.

More recently, these questions regained in interest and two articles, \cite{song2019high} and \cite{maecki2019universality} tackled more general problems. They got interested into the empirical (spectral) measure of the first introduced in \cite{graczyk2011multidimensional} general matrix valued stochastic differential equations on $\mathcal{H}_n$, the space of Hermitian $n\times n$ matrices, of the form
\begin{eqnarray*}
 dX_t=g(X_t)dW_t h(X_t)+h(X_t)d(W_t)^\dagger g(X_t)+\frac{1}{n}b_n(X_t)dt\/,\quad X_0\in \mathcal{H}_n\/,
\end{eqnarray*}
where $\dagger$ is the conjugate transpose operator, and where the continuous functions $g,h,b_n:\mathbb{R}\to\mathbb{R}$ act spectrally on $X_t$, i.e. if $X\in\mathcal{H}_n$, $U$ unitary and $D$ diagonal are such that $X=U^\dagger DU$, then $g$ is identified with the map $\mathcal H_n\ni X\mapsto U^\dagger\operatorname{Diag}(g(D_{1,1}),\cdots,g(D_{n,n}))U$. Here $W=(W_t)_t$ stands for a $n\times n$ complex-valued Brownian motion, i.e. the matrix valued process with entries being independent one-dimensional complex-valued Brownian motions. This general formulation encompasses the $\beta$-Wishart (or $\beta$-Laguerre) model and  the $\beta$-Jacobi model, as it is stated in \cite{graczyk2011multidimensional}. After proving the tightness of the family of measure valued processes, these two papers use a \cite{MR1217451}-like method to  get  the integro-differential equation that is  satisfied by the accumulation points of the measure valued processes family. In \cite{maecki2019universality}, the more general of the two articles, the authors derive existence and uniqueness of the solution to the integro-differential equation, and thus the convergence of the empirical measure valued process when $n\rightarrow\infty$ using the fact that the moments of the limit measure valued flow are uniquely determined by the integro-differential equation. Moreover, in some peculiar cases, they give an explicit expression of the limit, but only for the initial condition  $X_0=0$. In \cite{song2019high}, rather than using a moment approach, they translate as in \cite{MR1217451} the integro-differential equation problem into a Burgers equation, which they do not solve in general, but only for the initial condition $X_0=0$, using scaling properties of the initial SDE problem. For this initial condition, the Marcenko-Pastur distribution appears naturally in the expression of the measure valued flow limit.

The $\beta$-Wishart (or $\beta$-Laguerre) problem was also tackled in \cite{Allez}. They compute there in the  high temperature regime (when $\beta n\underset{n\rightarrow \infty}{\longrightarrow}2c\in(0,\infty)$)  the limit integro-differential equation, parametrized by $c$, the stationary probability measure of this equation, and remark that this limit distribution interpolates the Marcenko-Pastur distribution on the  $c\rightarrow+\infty$ limit and the Gamma distribution on the  $c\rightarrow0$ limit.

In \cite{trinh2020betalaguerre} and \cite{trinh2020betajacobi} were studied the same questions in the peculiar cases of respectively $\beta$-Wishart (or $\beta$-Laguerre) and $\beta$-Jacobi processes using a moment based method as in \cite{maecki2019universality} to prove the uniqueness the limit measure process. The authors moreover compute, in the   high temperature regime (when $\beta n\underset{n\rightarrow \infty}{\longrightarrow}2c\in(0,\infty)$), the long-time behaviour of the limit measure flow for both process families.

\paragraph{}
The interest of this note is to understand the limit behaviour of the empirical spectral measure of $\frac{1}{m}M_tM_t^*$ and $\sqrt{\frac{1}{m}M_tM_t^*}$ both in a real and a complex framework, using rectangular free probability tools. This point of view, more random matrix theory and free probability oriented,  sheds a new light on the problem of the limit behaviour of the empirical  measure process of $\beta$-Wishart particles for all $\beta>0$, and allows to compute the limit measure valued flow for any initial condition, and to recover the long-time behaviour of the limit measure valued flow.  For the sake of completeness, we also derive the integro-differential equation and the complex PDE approach of the problem. On top of that, an interest of Theorem  \ref{th:betamfW} is, by computing the Cauchy-Stieltjes transform of the $\mu^{n,W}$ (the reader will find a definition of this trasform in Section \ref{sect:preliminaries} and a definition of this probability measure in Section \ref{sect:generalwishart}), to provide a numerical approach in solving complex Burgers partial differential equations of the form \eqref{PDEBurgers}.
\paragraph{}
The note is organized the following way. After the Introduction where the main results are given, the complex version of these results are given in Section 2. The reader will find in Section 3 an application of Theorem \ref{th:mfl} to the general $\beta$-Wishart processes. Preliminary results about the Cauchy-Stieltjes transform and about the free and the rectangular free convolution are given in Section 4, and the results are proven in Section 5.
\paragraph{Acknowledgement} : I thank Djalil Chafai and Benjamin Jourdain for numerous fruitful discussions.

\section{The complex framework}
\label{sect:complex}

Theorem \ref{th:mfl} in fact true in a complex framework, i.e. for $(M_t)_{t\ge0}$ being a complex valued matrix process, following the SDE
\[dM_t=\kappa dW_t-\gamma M_tdt,\] with $W$ a $n\times m$ matrix filled with  independent complex Brownian motions (of the form $(B^1_t+\mathrm{i}B^2_t)_{t\geq0}$ where $(B^1_t)_{t\geq0}$ and $(B^2_t)_{t\geq0}$ are independent real Brownian motions), $M_0$ a complex valued random matrix. In this Section, we keep the definitions of the  $(\lambda^{1,n}_t,\dots, \lambda^{n,n}_t)_t$ and $(x^{1,n}_t,\dots, x^{n,n}_t)_t$ as the eigenvalues of respectively $\frac{1}{m}M_tM_t^*$ and    $\sqrt{\frac{1}{m}M_tM^*_t}$   but with the symbol $*$ denoting the conjugate transpose  rather than the transpose operator.

\begin{theorem}[Complex case]
\label{th:mflcomplex}
Theorem \ref{th:mfl} applies in this framework with for all $t\geq0$ :
  \begin{equation*}
      \sigma_t^2 =\left\{
      \begin{aligned}
      &\frac{\kappa^2}{\gamma}(1-e^{-2\gamma t}) \quad\text{if}\quad \gamma\neq0\\
      &2\kappa^2t \quad\text{if}\quad \gamma=0.
      \end{aligned}
      \right.
  \end{equation*}
  and with $\sigma_\infty = \sqrt{\frac{\kappa^2}{\gamma}}$.
\end{theorem}
Let us note that in this context, the eigenvalues $(\lambda^{1,n}_t,\dots, \lambda^{n,n}_t)_t$ follow the SDE (see \cite{MR1871699}) : 

\begin{equation*}
    d\lambda^{i,n}_t=\frac{2\kappa }{\sqrt{m}}\sqrt{\lambda^{i,n}_t}dB^i_t+2\kappa^2dt-2\gamma\lambda^{i,n}_tdt+\frac{2\kappa^2 }{m}\sum_{j\neq i}\frac{\lambda^{i,n}_t+\lambda^{j,n}_t}{\lambda^{i,n}_t-\lambda^{j,n}_t}dt \text{ for all }i\in\{1,\dots n\}
\end{equation*}
\[0\leq\lambda^{1,n}_t<\dots<\lambda^{n,n}_t\text{ a.s. dt-a.e.},\]
where $B^1,\dots, B^n$ are independent Brownian motions. It corresponds to a $\beta$-Wishart process with $\beta=2$, or to a $W(2,1)$ process, which is defined in the next Section.

\section{Application to more general Wishart processes}
\label{sect:generalwishart}

For $\beta_1,\beta_2>0$, let us consider the changed of time \emph{general Wishart process} defined by the SDE :
 \begin{equation}
  \label{eq:general}
       d\lambda^{i,n,W}_t=\frac{2\kappa }{\sqrt{m}}\sqrt{\lambda^{i,n,W}_t}dB^i_t-2\gamma\lambda^{i,n,W}_tdt+\beta_1\kappa^2\left(1+\frac{\beta_2}{m}\sum_{j\neq i}\frac{\lambda^{i,n,W}_t+\lambda^{j,n,W}_t}{\lambda^{i,n,W}_t-\lambda^{j,n,W}_t}\right)dt \text{ for all }i\in\{1,\dots n\}
  \end{equation}
  \[0\leq\lambda^{1,n,W}_t<\dots<\lambda^{n,n,W}_t\text{ a.s. dt-a.e.},\]
where $B^1,\dots, B^n$ are independent real Brownian motions. We will refer to this SDE by $W(\beta_1,\beta_2)$. As stated in the introduction, the classical $\beta$-Wishart processes found in the literature are of the form $W(1,\beta)$ or $W(\beta,1)$ (after the change of variables $t\mapsto\frac{m}{\kappa^2}t$). We also remark that the SDE (\ref{eds}) corresponds to $W(1,1)$. Results from \cite{MR3296535} and  \cite{JK} together show (after applying the change of variables $t\mapsto\frac{m}{\kappa
^2}t$ to the SDE (\ref{eq:general})) that this SDE admits a strong pathwise unique solution defined on $\mathbb{R}_+$ as soon as $m-(n-1)\beta_2>0$ and 
\begin{itemize}
    \item $\beta_1\beta_2\geq1$ ;
    \item or if $0<\beta_1\beta_2<1$ when $m\beta_1+(2-n)\beta_1\beta_2\geq1$.
\end{itemize}

Let us note that the long time behaviour of such a process was studied before : the change of variables $t\mapsto\frac{m}{\kappa^2}t$ in the SDE (\ref{eq:general}) coupled with the results \cite[Lemma 3.1 and Proposition 2.8]{JK}  prove that, given integrable initial conditions, i.e. $\sum_{i=1}^n\mathbb{E}[\lambda^{i,n,W}_0]<+\infty$, the  distribution of $(\lambda^{1,n,W}_t,\dots,\lambda^{n,n,W}_t)$  converges weakly when $t\rightarrow+\infty$ to a unique stationary probability measure with density with respect to the Lebesgue measure 
\begin{equation}
    \label{statdistrib}
    (\lambda^1,\dots,\lambda^n) \mapsto \frac{1}{\mathcal{Z}}\prod_{i=1}^n\left((\lambda^i)^{\frac{\beta_1m-(n-1)\beta_1\beta_2}{2}-1}e^{-\frac{m\gamma}{\kappa^2}\lambda^i}\prod_{j\neq i}|\lambda^j-\lambda^i|^{\beta_1\beta_2/2}\right)\mathds{1}_{0\leq\lambda^1\leq\dots\leq\lambda^n}, 
\end{equation}
where $\mathcal{Z}$ is a normalizing constant. For $\beta_2=1$, this Gibbs measure is the $\beta$-Laguerre ensemble, on real symmetric matrices for $\beta_1=1$, on complex hermitian matrices for $\beta_1=2$ and on quaternion self-dual matrices for $\beta_1=4$, see for instance \cite{MR2641363}. This distribution is moreover related to the distribution of the singular values  of $n\times m$ random matrix with independent identically distributed real (for $\beta_1=1$) or complex (for $\beta_2=2$) Gaussian entries. 

It is a well known fact that, for $\beta_2=1$, if $(\lambda^1,\dots,\lambda^n)$ is a random vector distributed according to (\ref{statdistrib}), then 
\begin{equation}
    \label{prop:stat}
    \nu^{n,W}=\frac{1}{n}\sum_{i=1}^n\delta_{\lambda^i}\underset{\underset{n/m\rightarrow\alpha}{n\rightarrow\infty}}{\longrightarrow}\mu^{MP_{\alpha,\sigma_\infty}},
\end{equation}
with $\sigma_\infty=\sqrt{\frac{\beta_1\kappa^2}{2\gamma}}$,
see for instance \cite[Theorem 5.5.7]{hiaipetz}.

Let us define for all $t\geq0$ : \[\nu^{n,W}_t=\frac{1}{n}\sum_{i=1}^n\delta_{{\lambda^{i,n,W}_t}},\qquad \mu^{n,W}_t=\sqrt{\nu^{n,W}_t}=\frac{1}{2n}\sum_{i=1}^n\left(\delta_{\sqrt{\lambda^{i,n,W}_t}}+\delta_{-\sqrt{\lambda^{i,n,W}_t}}\right).\]

The next result tackles the limit  of the sequence $(\nu_t^{n,W})_n$. As written in the introduction, it is a peculiar case of the results of \cite{maecki2019universality} reproduced here partly for the sake of completeness. The limit of the empirical measure problem is transformed into a complex PDE problem, a complex Burgers equation, to which the Cauchy-Stieltjes transform of the limit measure process, if it exists, must be solution.

Definitions and properties about the Cauchy-Stieltjes transform are given in Section 3.

\begin{theorem}[Complex Burgers and mean field limit : weak formulation]
\label{th:burgers}
Let us assume that the measure valued sequence $(\nu_0^{n,W})_n$ converges weakly in probability when $n\rightarrow+\infty$ with $\frac{n}{m}\rightarrow\alpha\in(0,1]$ to a limit probability measure denoted by $\nu^W_0$ and that
\[\underset{n}{\sup}\int x^8\nu^{n,W}_0(dx)<\infty.\]
Then the family $\{(\nu^{n,W}_t)_{t\ge0},n\in\mathbb{N}^*\}$ is tight and any limiting measure valued flow when $n$ goes to infinity with $\frac{n}{m}\rightarrow\alpha\in(0,1]$ satisfies for  all twice continuously differentiable real test function $f$ the  equation : 
 
\begin{equation}
    \langle \nu_t,f\rangle =  \langle \nu_0,f\rangle + \int_0^t\langle\nu_s,(\beta_1\kappa^2-2\gamma\Phi) f'\rangle ds + \frac{\alpha\beta_1\beta_2\kappa^2 }{2}\int_0^t\left(\iint(x+y)\frac{f'(x)-f'(y)}{x-y}\nu_s(dx)\nu_s(dy)\right)ds\label{closedequ}
\end{equation} 
where $\Phi:x\mapsto x$ and with the convention $\frac{f'(x)-f'(y)}{x-y}=f''(x)$ when $x=y$. If $\nu_0$ admits a characteristic function, and if this function is analytic on a neighborhood of the origin, then this equation admits a unique solution.

Moreover, if $(\nu_t)_{t\ge0}$ is a solution to equation (\ref{closedequ}) and if for all $t\geq0$, $G_t$ is the Cauchy-Stieltjes transform of $\nu_t$, then $G$ satisfies the complex Burgers PDE
\begin{equation}
\label{PDEBurgers}
    \begin{split}
        \frac{\partial}{\partial t}G_t(z) &=(\alpha\beta_1\beta_2\kappa^2 -\beta_1\kappa^2+2\gamma z)\frac{\partial}{\partial z}G_t(z) - 2\alpha\beta_1\beta_2\kappa^2  zG_t(z)\frac{\partial}{\partial z}G_t(z)-\alpha\beta_1\beta_2\kappa^2  G_t^2(z)+2\gamma G_t(z),\\
        G_0(z)&=\int\frac{\nu_0^W(dv)}{z-v}=\phi(z).
    \end{split}
\end{equation}
\end{theorem}

\begin{remark}[Moments dynamics]
\label{rmk:moments}
Let $\nu=(\nu_t)_t$ be a real measure valued flow verifying equation (\ref{closedequ}), and let us denote for all $k\in\mathbb{N}^*,t\geq0$ : $m^k_t=\mathbb{E}[ \nu_t^{k}]$. Then for all $t\ge0$:
\begin{enumerate}[label=(\roman*)]
    \item $m_t^1  =  \frac{\beta_1\kappa^2}{2\gamma}+\left(m_0^1-\frac{\beta_1\kappa^2}{2\gamma}\right)\mathrm{e}^{-2\gamma t}$;
    \item $ m_t^2=\frac{1}{4\gamma^2}\Big[(\beta_1^2\kappa^4+(\alpha \beta_1\beta_2\kappa^2 -4\gamma m_0^1)\beta_1\kappa^2+4\gamma^2m_0^2-4\alpha \beta_1\beta_2\kappa^2 \gamma m_0^1)\mathrm{e}^{-4\gamma t}$
    
    $\qquad\qquad+4\beta_1\kappa^2(\beta_2\alpha  +1)\left(\left(\gamma m_0^1-\frac{\beta_1\kappa^2}{2}\right)\mathrm{e}^{-2\gamma t}+\frac{\beta_1\kappa^2}{4}\right)\Big]$;
    \item $\forall k\geq3,$ $\frac{\mathrm{d}}{\mathrm{dt}} m_t^k=k\beta_1\kappa^2 m_t^{k-1}-2k\gamma m_t^{k}+\frac{k\alpha\beta_1\beta_2 \kappa^2 }{2}\sum_{j=0}^{k-2}\left( m_t^{j+1} m_t^{k-2-j}+m_t^{j} m_t^{k-j-1}\right).$
\end{enumerate}
These equations are used in \cite{maecki2019universality, trinh2020betalaguerre} to prove existence and uniqueness of the solution to the integro-differential equation (\ref{closedequ}).
\end{remark}
\begin{remark}
The PDE (\ref{PDEBurgers}) satisfies the assumptions of the Cauchy-Kowalevski Theorem, see for instance \cite[Theorem 1.25]{folland1995introduction}. However, this theorem only gives local existence and uniqueness of an analytic solution (both in space and time) to the PDE.  We did not manage to expand this approach to conclude to global existence and uniqueness of a solution.
\end{remark}

The next Proposition gives results on the PDE (\ref{PDEBurgers}) using a different method than the ones used in \cite{maecki2019universality}. 
\begin{proposition}[Marcenko Pastur stability along the Burgers PDE and stationarity]
\label{propoburgers} 
\ 
\begin{enumerate}[label=(\roman*)]
    \item Let $\rho_0,\sigma_0\in\mathbb{R}_+$. Let us assume $\kappa\neq0$ and that $\nu_0^W=\underset{n\rightarrow\infty}{\lim}\nu_0^{n,W}$ follows a Marcenko-Pastur distribution of shape parameter $\rho_0$ and of scale parameter $\sigma_0$.
    
    Then, the PDE (\ref{PDEBurgers}) admits a solution which is the Cauchy-Stieltjes transform of a Marcenko-Pastur distribution  for all $t>0$ if and only if $\rho_0=\beta_2\alpha$. In this case, this solution can be written as the Cauchy-Stieltjes transform of the Marcenko-Pastur distribution $\mu^{MP_{\beta_2\alpha,\sigma(t)}}$ with
    \begin{equation*}
        \sigma:t\in\mathbb{R}_+\mapsto \left\{
        \begin{aligned}
            &\sqrt{\left(\sigma^2_0-\frac{\beta_1\kappa^2}{2\gamma}\right)\mathrm{e}^{-2\gamma t}+\frac{\beta_1\kappa^2}{2\gamma}} \quad\text{if}\quad \gamma\neq0\\
            &\sqrt{\sigma_0^2+\beta_1\kappa^2t} \quad\text{if}\quad \gamma=0
        \end{aligned}
        \right.
    \end{equation*}
    \item Let us assume $\gamma\neq0$ and let $\sigma_0=\frac{\sqrt{\beta_1}\kappa}{\sqrt{2\gamma}}$. Then $G_{\mu^{MP_{\beta_2\alpha,\sigma_0}}}$ is the unique stationary solution to the equation  (\ref{PDEBurgers}) corresponding to the Cauchy-Stieltjes transform of a probability measure on $\mathbb{R}$. Moreover, under the assumptions of Theorem \ref{th:burgers} and if $\nu_0=\mu^{MP_{\beta_2\alpha,\sigma_0}}$, then for all $t\geq0$, $\nu^W_t=\underset{n\rightarrow\infty}{\lim}\nu_t^{n,W}=\mu^{MP_{\beta_2\alpha,\sigma_0}}$.
\end{enumerate}
\end{proposition}

To our knowledge, the assertion $(i)$ was not remarked yet in the past literature, and $(ii)$ can be seen as in the continuity of the results of \cite{Allez}. Indeed, their computations are  in the  high temperature regime (when $\beta n\underset{n\rightarrow \infty}{\longrightarrow}2c\in(0,\infty)$) so that the limit of their integro-differential equation and its stationary probability measure are parametrized by $c$,  and they  remark that this limit stationary probability measure converges to the Marcenko-Pastur distribution in the limit $c\rightarrow+\infty$.

The next result gives a complete answer to the limit of $(\mu^{n,W})_n$  problem, using free probability tools (in particular the rectangular free convolution defined in Section 3). It is the analog of Theorem \ref{th:mfl} in this context.

\begin{theorem}[Mean-field limit for the general Wishart process]
\label{th:betamfW}
Under the assumptions of Theorem \ref{th:burgers}, let us moreover assume $\beta_2\alpha\leq1$, that $(\lambda^{1,n,W}_t,\dots,\lambda^{n,n,W}_t)_t$ is defined on $\mathbb{R}_+$ and   that  $\mu_0^W=\sqrt{\nu_0^W}$. Let us moreover assume that $\nu_0^W$ has a characteristic function which is analytic on a neighborhood of the origin. Then for all $t\geq0$,
  ${(\mu^{n,W}_t)}_{n\geq0}$ converges weakly in probability and
  \[
   \mu^W_t:=\lim_{n\to\infty}\mu^{n,W}_t=(e^{-\gamma t}\mu_0^W)\boxplus_{\beta_2\alpha} \sqrt{\mu^{\mathrm{MP}_{\beta_2\alpha, \sigma_t}}}
  \]
  where $\boxplus_{\beta_2\alpha}$ is the rectangular free convolution of  parameter $\beta_2\alpha$, where $e^{-\gamma t}\mu_0^W$ denotes the  push-forward of $\mu_0^W$ by the map $v\mapsto e^{-\gamma t}v$,  and where for all $t\geq0$ :
  \begin{equation*}
      \sigma_t^2 =\left\{
      \begin{aligned}
      &\frac{\beta_1\kappa^2}{2\gamma}(1-e^{-2\gamma t}) \quad\text{if}\quad \gamma\neq0\\
      &\beta_1\kappa^2t \quad\text{if}\quad \gamma=0.
      \end{aligned}
      \right.
  \end{equation*}
 Moreover, if $\gamma\neq0$ and with $\sigma_\infty = \sqrt{\frac{\beta_1\kappa^2}{2\gamma}}$,
 \[\underset{t\rightarrow+\infty}{\lim}\mu^W_t = \sqrt{\mu^{\mathrm{MP}_{\beta_2\alpha,\sigma_\infty  }}}.\]

\end{theorem}

Theorem \ref{th:betamfW} coupled with the limit  (\ref{prop:stat}) allows to show the following result.

\begin{theorem}[Commutativity of the limits]\label{th:limcommutW}
Under the assumptions of Theorem \ref{th:betamfW}, let us moreover suppose $\gamma\neq0,\beta_2=1$ and that for all $n\in\mathbb{N}^*$: $\sum_{i=1}^n\mathbb{E}[\lambda^{i,n,W}_0]<+\infty$. Let us note $\sigma_\infty=\sqrt{\frac{\beta_1\kappa^2}{2\gamma}}$. Then we have in the sense of weak convergence 

\[\underset{\underset{n/m\rightarrow\alpha}{n\rightarrow\infty}}{\lim} \underset{t\rightarrow\infty}{\lim}\nu^{n,W}_t = \underset{t\rightarrow\infty}{\lim} \underset{\underset{n/m\rightarrow\alpha}{n\rightarrow\infty}}{\lim} \nu^{n,W}_t = \mu^{MP_{\alpha,\sigma_\infty}}.\]

\end{theorem}

\section{Useful tools}
\label{sect:preliminaries}
\subsection{The Cauchy-Stieltjes transform}
Let us denote \[\mathbb{C}_\pm= \{z\in\mathbb{C}, \pm\Im(z)>0\}.\]
Let $\mu$ be a probability measure on $\mathbb{R}$. Its Cauchy-Stieltjes transform is defined by
\begin{equation*}
    z\in\mathbb{C}_+\mapsto G_\mu(z)=\int_\mathbb{R}\frac{\mu(dv)}{z-v}\in\mathbb{C}_-.
\end{equation*}
The next result shows that this transformation characterizes the probability measure on $\mathbb{R}$.

\begin{theorem}[\cite{mingospeicher}]\label{thmcauchytransform}
(1)Let $\mu$ be a probability measure on $\mathbb{R}$. Then,
\begin{enumerate}[label=(\roman*)]
\item $G_\mu$ is analytic on $\mathbb{C}_+$,
\item we have \[\underset{y\rightarrow\infty}{\lim}\mathrm{i}yG_\mu(\mathrm{i}y)=1.\]
\end{enumerate}

(2) Any probability measure on $\mathbb{R}$ can be recovered from its Cauchy-Stieltjes transform $G_\mu$ via the Stieltjes inversion formula : for all $a,b\in\mathbb{R}$ with a<b,
\begin{equation*}
    -\underset{\epsilon\downarrow0}{\lim}\frac{1}{\pi}\int_a^b\Im(G_\mu(x+\mathrm{i}\epsilon))dx=\mu((a,b))+\frac{1}{2}\mu(\{a,b\}).
\end{equation*}
It follows that the Cauchy-Stieltjes transform $G_\mu$ characterises the measure $\mu$.

(3)Let $G:\mathbb{C_+\mapsto\mathbb{C}_-}$ be an analytic function which satisfies \[\underset{y\rightarrow\infty}{\limsup}y|G(\mathrm{i}y)|=1.\] Then, there exists a unique probability measure $\mu$ on $\mathbb{R}$ such that $G=G_\mu$.
\end{theorem}

\subsection{Free convolution and rectangular free convolution from a random matrix point of view}
\label{subsect:freeconvpreli}
Let $M$ be a $n\times n$ real symmetric matrix. We will denote its eigenvalues by \[\lambda^1(M)\leq\dots\leq\lambda^n(M),\] and its empirical (spectral) measure by
\[\mu^M=\frac{1}{n}\sum_{i=1}^n\delta_{\lambda^i(M)}.\]
Let $N$ be a $n\times m$ real valued matrix. Its empirical singular measure is the empirical measure of the positive semi-definite symmetric matrix $NN^*$.

Free convolutions are operations on probability measures on the real line which allow to compute the empirical spectral  or singular measures of large random matrices (i.e. matrices whose size goes to infinity) which are expressed as sums or products of independent large random matrices whose spectral measures are already known. The free additive convolution is denoted by $\boxplus$. 
\begin{theorem}[Additive Free Convolution for random matrices \cite{voiculescu1992free}]
For all $n\in\mathbb{N}^*$ let us define $M_n$ and $N_n$ two independent $n\times n$ random symmetric matrices, such that 
\begin{itemize}
    \item the distribution of $M_n$ is invariant under the action of the unitary group by conjugation, 
    \item the empirical measures sequences $(\mu^{M_n})_n$ and $(\mu^{N_n})_n$ converge weakly almost surely when $n$ goes to infinity to non-random probability measures, respectively $\mu^M_\infty$ and $\mu^N_\infty$.
\end{itemize}
Then, in the sense of weakly almost sure convergence, \[\mu^{M_n+N_n}\underset{n\rightarrow\infty}{\longrightarrow}\mu^M_\infty\boxplus\mu^N_\infty.\]
\end{theorem}
This operation can be equivalently defined in reference to free elements of a non commutative probability space.
Let $\mu$ be a probability measure on $\mathbb{R}$. Its R-transform is defined on a neighbourhood of zero by
\[R_\mu(z) = G_\mu^{-1}(z)-\frac{1}{z}.\]

\begin{theorem}
The R-tranform linearizes the free convolution : for $\mu$ and $\nu$ probability measures on the real line, and for $z$ in a neighbourhood of zero,
\[R_{\mu\boxplus\nu}(z)=R_{\mu}(z)+R_{\nu}(z),\] and $\mu\boxplus\nu$ is the unique probability measure verifying this relation.
\end{theorem}
Let us denote by $\mathcal{B}(\mathbb{R})$ the Borel sets of $\mathbb{R}$.
Let us define the application from $\mathcal{M}_1(\mathbb{R}_+)$, the set of probability measures on $\mathbb{R}_+$, to $\mathcal{M}^S_1(\mathbb{R}) = \{\mu\in\mathcal{M}_1(\mathbb{R}), \forall A\in\mathcal{B}(\mathbb{R}), \mu(A)=\mu(-A)\}$, the set of symmetric probability measures on $\mathbb{R}$: \[\mathrm{sym}:\mu\in\mathcal{M}_1(\mathbb{R}_+)\mapsto\left(\mathrm{sym}(\mu):A\in\mathcal{B}(\mathbb{R})\mapsto\frac{\mu(A\cap\mathbb{R}_+)+\mu(-A\cap\mathbb{R}_+)}{2}\right)\in\mathcal{M}
^S_1(\mathbb{R}).\] This application is bijective from $\mathcal{M}_1(\mathbb{R}_+)$ to $\mathcal{M}^S_1(\mathbb{R})$ and admits the inverse: 
\begin{equation}
\label{eq:syminv}
    \mathrm{sym}^{-1}(\nu): A\in\mathcal{B}(\mathbb{R}_+)\mapsto \mathds{1}_A(0)\nu(\{0\})+2\nu(A\backslash\{0\} ).
\end{equation}

For any $\alpha\in[0,1]$, the rectangular free convolution denoted $\boxplus_\alpha$ can be defined the following way.

\begin{theorem}[Additive free rectangular convolution of ratio $\alpha$ for random matrices \cite{rect}]
\label{th:rectfreeconv}
For all $n,m\in\mathbb{N}^*$ let us define $M_{n,m}$ and $N_{n,m}$ two independent $n\times m$ random  matrices, such that 
\begin{itemize}
    \item the distribution of $M_{n,m}$ is invariant under the action of the unitary group by conjugation on any side, 
    \item the empirical measures sequences $(\mu^{\sqrt{M_{n,m}M_{n,m}^*}})_{n,m}$ and $(\mu^{\sqrt{N_{n,m}N_{n,m}^*}})_{n,m}$ respectively converge in probability, when $n$ and $m$ go to infinity with $n/m$ tending to $\alpha\in(0,1]$, to non-random probability measures  $\mu^M_\infty$ and $\mu^N_\infty$.
\end{itemize}
Then in the sense of  weak convergence in probability, \[\mathrm{sym}\left(\mu^{\sqrt{(M_{n,m}+N_{n,m})(M_{n,m}+N_{n,m})^*}}\right)\underset{\underset{n/m\rightarrow\alpha}{n,m\rightarrow\infty}}{\longrightarrow}\mathrm{sym}(\mu^M_\infty)\boxplus_\alpha \mathrm{sym}(\mu^N_\infty).\]
\end{theorem}
This operation can also be equivalently defined in reference to free elements of a  rectangular non commutative probability space.

Let $\mu$ be a symmetric probability measure on $\mathbb{R}$. Its rectangular Cauchy transform with ratio $\alpha$ is defined by
\[H_\mu:z\in\mathbb{C}\backslash[0,+\infty)\mapsto  z(\alpha M_{\mu^2}(z)+1)( M_{\mu^2}(z)+1),\]
where \[M_{\mu^2}:z\in\mathbb{C}\backslash[0,+\infty)\mapsto \int_\mathbb{R}\frac{zv^2}{1-zv^2}d\mu(v)=\frac{1}{\sqrt{z}}G_\mu\left(\frac{1}{\sqrt{z}}\right)-1.\]
This equality can be derived  the following way in a neighbourhood of zero :
\begin{align*}
    G_{\mu}\left(\frac{1}{\sqrt{z}}\right) &= z^{\frac{1}{2}}\int_\mathbb{R}\frac{1}{1-z^{\frac{1}{2}}v}d\mu(v)\\
    &= z^{\frac{1}{2}}\sum_{k\geq0}\int_\mathbb{R}z^{\frac{k}{2}}v^kd\mu(v)\\
    &= z^{\frac{1}{2}}\left(1+zv^2\sum_{k\geq0}\int_\mathbb{R}z^{k}v^{2k}d\mu(v)\right),
\end{align*}
where we use in the last equality the fact that $\mu$ is symmetric.

The rectangular R-transform with ratio $\alpha$ of $\mu$ is defined on a neighbourhood of zero by
\[C_\mu(z)=U\left(\frac{z}{H^{-1}_\mu(z)}-1\right),\] where on a neighbourhood of zero
\begin{equation*}
U(z)=
    \left\{
\begin{array}{c }\displaystyle
    \frac{-\alpha-1+[(\alpha+1)^2+4\alpha z]^{1/2}}{2\alpha} \text{ if }\alpha\neq0 \\
    z \text{ if } \alpha=0. 
\end{array}
\right.
\end{equation*}

\begin{theorem}[\cite{rect}]
\label{th:rectRtransform}
The rectangular R-transform with ratio $\alpha$ linearizes the rectangular free convolution with ratio $\alpha$ : for $\mu$ and $\nu$ symmetric probability measures on the real line, and for $z$ in a neighbourhood of zero,
\[C_{\mu\boxplus_\alpha\nu}(z)=C_{\mu}(z)+C_{\nu}(z),\] and $\mu\boxplus_\alpha\nu$ is the unique symmetric probability measure verifying this relation.
\end{theorem}

\begin{theorem}[Injectivity of the rectangular R-transform, \cite{rect}]
\label{th:injectrectRtransform}
If the rectangular R-transforms with ratio $\alpha$ of two symmetric probability measures coincide on a neighbourhood of $0$ in $(-\infty,0]$, then the measures are equal.
\end{theorem}

The convergence results with free convolution and rectangular free convolution overlap in the case of $M_n,N_n$ $n\times n$ square symmetric semi-definite positive codiagonalizable matrices. Then we have

\[
		\begin{matrix}
			\mathrm{sym}\left(\mu^{\sqrt{(M_{n}+N_{n})(M_{n}+N_{n})^*}}\right) &= &\mathrm{sym}(\mu^{M_n+N_n}) \\
			~~~~~~~~~\Big\downarrow n \to \infty	&	&~~~~~~~~~\Big\downarrow n \to \infty	\\
			\mathrm{sym}(\mu^M_\infty)\boxplus_1 \mathrm{sym}(\mu^N_\infty)	& =	& \mathrm{sym}(\mu^M_\infty\boxplus\mu^N_\infty).
		\end{matrix}
		\]

\section{Proofs}

\begin{proof}[Proof of Theorem \ref{th:mfl}]
The entries of the matrix $M$  are independent Ornstein-Uhlenbeck processes just as the one considered in \cite{MR1132135}. We thus can write for all $t\geq 0$
\begin{equation*}
    M_t=M_0e^{-\gamma t}+\kappa\int_0^te^{\gamma(s-t)}dW_s.
\end{equation*}

  The idea is to use a suitable rectangular free convolution, see
  \cite{rect}. More precisely, the matrices $A_{n,t}=M_0e^{-\gamma t}$ and $B_{n,t}=\kappa\int_0^te^{\gamma(s-t)}dW_s$ are independent and we may use a version of Voiculescu asymptotic
  freeness theorem, see \cite{MR1746976}. The precise result to use is  Theorem \ref{th:rectfreeconv} (see \cite[Theorem 3.13]{rect}). Indeed, $B_{n,t}$ is a matrix filled with i.i.d Gaussian random variables of variance
  \begin{equation*}
        \sigma^2_t =\mathbb{E}\left[\left|\kappa\int_0^te^{-\gamma(s-t)}dW_s\right|^2\right]=\left\{
        \begin{aligned}
            &\kappa^2\int_0^te^{2\gamma(s-t)}ds = \frac{\kappa^2}{2\gamma}(1-e^{-2\gamma t})  \quad\text{if}\quad \gamma\neq0\\
            &\kappa^2t \quad\text{if}\quad \gamma=0
        \end{aligned}
        \right.
    \end{equation*}
   by Ito's isometry, and is thus bi-unitary invariant.
  The Marcenko-Pastur theorem (see for instance \cite[Theorem 3.10]{baisilver}) tells us that in the sense of convergence in probability,
  \[\mu^{\frac{1}{m}B_{n,t}B_{n,t}^*}=\frac{1}{n}\sum_{i=1}^n\delta_{\lambda^{i}(\frac{1}{m}B_{n,t}B_{n,t}^*)}\underset{n\rightarrow\infty}{\longrightarrow}\mu^{MP_{\alpha, \sigma_t}},\] weakly, and thus we have :
  \[\mu^{\sqrt{\frac{1}{m}B_{n,t}B_{n,t}^*}}=\frac{1}{n}\sum_{i=1}^n\delta_{\lambda^{i}(\sqrt{\frac{1}{m}B_{n,t}B_{n,t}^*})}\underset{n\rightarrow\infty}{\longrightarrow}\sqrt{.}\sharp\mu^{MP_{\alpha, \sigma_t}},\] which is a non random measure.
  
 An application of Theorem \ref{th:rectfreeconv} to $A_{n,t}$ and $B_{n,t}$ for all $t>0$ ends the first part of the proof.
 
Let us assume $\gamma\neq0$. According to \cite[Theorem 2.12]{rect}, the binary operation $\boxplus_\alpha$  is continuous (with respect to the weak convergence) on the set of symmetric probability measures on the real line, and so does the the rectangular R-transform $C$. Thus, as
 
 \[\underset{t\rightarrow+\infty}{\lim}\mathrm{sym}(e^{-\gamma t}\mu_0)=\delta_0 \quad \text{ and }\quad \underset{t\rightarrow+\infty}{\lim}\sqrt{\mu^{\mathrm{MP}_{\alpha, \sigma_t}}}=\sqrt{\mu^{\mathrm{MP}_{\alpha, \sigma_\infty}}},\]
 we have
  \[\underset{t\rightarrow+\infty}{\lim}\mathrm{sym}(\mu_t) = \delta_0\boxplus_\alpha\sqrt{\mu^{\mathrm{MP}_{\alpha, \sigma_\infty}}}.\]
 Moreover, the formula in Subsection \ref{subsect:freeconvpreli} allows to compute for all $z\in\mathbb{C}$ : 
 \begin{align*}
     M_{(\delta_0)^2}(z)&=0,\\
     H_{\delta_0}(z)&=z,\\
     C_{\delta_0}(z) & = U(0)= 0.
 \end{align*}
 Consequently, applying Theorem \ref{th:rectRtransform}, 
 \[C_{\delta_0\boxplus_\alpha\sqrt{\mu^{\mathrm{MP}_{\alpha, \sigma_\infty}}}}=C_{\delta_0}+C_{\sqrt{\mu^{\mathrm{MP}_{\alpha, \sigma_\infty}}}} = C_{\sqrt{\mu^{\mathrm{MP}_{\alpha, \sigma_\infty}}}},\]
which allows to conclude, applying Theorem \ref{th:injectrectRtransform}, that 
\[\delta_0\boxplus_\alpha\sqrt{\mu^{\mathrm{MP}_{\alpha, \sigma_\infty}}}=\sqrt{\mu^{\mathrm{MP}_{\alpha, \sigma_\infty}}},\] which ends the proof.
\end{proof}

\begin{proof}[Proof of Theorem \ref{th:limcommut}]
Let us first show that if  $(\lambda^{1,n},\dots ,\lambda^{n,n})$ is a random vector distributed according to the distribution with density with respect to the Lebesgue measure:
\begin{equation}
    \label{eq:statdistribreal}
    (\lambda^1,\dots,\lambda^n) \mapsto \frac{1}{\mathcal{Z}}\prod_{i=1}^n\left((\lambda^i)^{\frac{m-n+1}{2}-1}e^{-\frac{m\gamma}{\kappa^2}\lambda^i}\prod_{j\neq i}|\lambda^j-\lambda^i|^{1/2}\right)\mathds{1}_{0\leq\lambda^1\leq\dots\leq\lambda^n},
\end{equation}
where $\mathcal{Z}$ is a normalizing constant, and if we define the empirical measure \[\nu^{n} =\frac{1}{n}\sum_{i=1}^n\delta_{\lambda^{i,n}},\] and note $\sigma_\infty=\sqrt{\frac{\kappa^2}{2\gamma}}$, then, in the sense of weak convergence,
\[\nu^{n}\underset{\underset{n/m\rightarrow\alpha}{n\rightarrow+\infty}}{\longrightarrow}\mu^{MP_{\alpha,\sigma_\infty}},\] which is a peculiar case of (\ref{prop:stat}).

Let us consider the $n\times m$ random matrix $M_0$ whose coordinates are independent identically distributed centered real  Gaussian random variables of variance $\sigma_\infty^2=\frac{\kappa^2}{2\gamma}$. An application of \cite[Proposition 7.4.1]{Pastur:2623028} shows that of the matrix $\frac{1}{m}M_0M_0^*$ follow the density (\ref{eq:statdistribreal}).
Thus, we have 
\[\mu^{\frac{1}{m}M_0M_0^*}=\nu^{n},\] in the sense of equality in law.

Moreover, an application of the Marcenko-Pastur theorem (see for instance \cite[Theorem 3.10]{baisilver}) shows that in the sense of convergence in probability
\[\mu^{\frac{1}{m}M_0M_0^*}\underset{\underset{n/m\rightarrow\alpha}{n\rightarrow+\infty}}{\longrightarrow}\mu^{MP_{\alpha,\sigma_\infty}},\]
weakly which allows to conclude.

For all $n\in \mathbb{N}^*$, an application of \cite[Lemma 3.1 and Proposition 2.8]{JK} shows that in the sense of weak convergence,
\[\nu^{n}_t\underset{t\rightarrow+\infty}{\longrightarrow}\nu^{n}.\] 

We proved earlier that we have in the sense of weak convergence
\[\nu^{n}\underset{\underset{n/m\rightarrow\alpha}{n\rightarrow+\infty}}{\longrightarrow}\mu^{MP_{\alpha,\sigma_\infty}}.\] 
Theorem \ref{th:mfl} gives the two other limits and concludes the proof.
\end{proof}

\begin{proof}[Proof of Theorem \ref{th:mflcomplex}]
The proof of this Theorem mimics the proof of Theorem \ref{th:mfl}.
\end{proof}

\begin{proof}[Proof of Theorem \ref{th:burgers}]

 For $f$ a twice continuously differentiable real test function,
 we have thanks to the SDE (\ref{eq:general}) : 
 \[d\langle\nu^{n,W}_t,f\rangle = \frac{2\kappa }{n\sqrt{m}}\sum_{j=1}^n f'(\lambda_t^{j, n,W})\sqrt{\lambda_t^{j, n,W}}dB_t^{j}+\left(\frac{2\kappa^2 }{nm}\sum_{j=1}^n\lambda_t^{j, n,W}f''(\lambda_t^{j, n,W})+\frac{1}{n}\sum_{j=1}^n(\beta_1\kappa^2-2\gamma\lambda_t^{j, n,W})f'(\lambda_t^{j, n,W})\right)dt\] \[+ \left(\frac{\beta_1\beta_2\kappa^2 }{nm}\sum_{j=1}^nf'(\lambda_t^{j, n,W})\sum_{k\ne j}\frac{\lambda_t^{j, n,W}+\lambda_t^{k, n,W}}{\lambda_t^{j, n,W}-\lambda_t^{k, n,W}} \right)dt.\]
 We have :
 
 \begin{eqnarray}
     \frac{1}{nm} \sum_{j=1}^nf'(\lambda_t^{j, n,W})\sum_{k\ne j}\frac{\lambda_t^{j, n,W}+\lambda_t^{k, n,W}}{\lambda_t^{j, n,W}-\lambda_t^{k, n,W}} & = & \frac{1}{2nm}\sum_{j=1}^n\sum_{k\ne j}(f'(\lambda_t^{j, n,W})-f'(\lambda_t^{k, n,W}))\frac{\lambda_t^{j, n,W}+\lambda_t^{k, n,W}}{\lambda_t^{j, n,W}-\lambda_t^{k, n,W}} \nonumber \\
     & = & \frac{n}{2m}\iint_{\{x\ne y\}}(f'(x)-f'(y))\frac{x+y}{x-y}\nu^{n,W}_t(dx)\nu^{n,W}_t(dy) \nonumber \\
     & = & \frac{n}{2m}\iint(x+y)\frac{f'(x)-f'(y)}{x-y}\nu^{n,W}_t(dx)\nu^{n,W}_t(dy) -\frac{1}{m}\int xf''(x)\nu^{n,W}_t(dx).\nonumber
 \end{eqnarray}
 
 Finally we have
\begin{align}
    d \langle \nu^{n,W}_t,f\rangle  & =  \frac{2\kappa }{n\sqrt{m}}\sum_{j=1}^n f'(\lambda_t^{j, n,W})\sqrt{\lambda_t^{j, n,W}}dB_t^{j}+\left(\frac{2\kappa^2 }{nm}\sum_{j=1}^n\lambda_t^{j, n,W}f''(\lambda_t^{j, n,W})+\frac{1}{n}\sum_{j=1}^n(\beta_1\kappa^2-2\gamma\lambda_t^{j, n,W})f'(\lambda_t^{j, n,W})\right)dt \nonumber \\
     & \quad+  \beta_1\beta_2\kappa^2  \left(\frac{n}{2m}\iint(x+y)\frac{f'(x)-f'(y)}{x-y}\nu^{n,W}_t(dx)\nu^{n,W}_t(dy) -\frac{1}{m}\int xf''(x)\nu^{n,W}_t(dx)\right)dt \nonumber \\
     & =  \frac{2\kappa }{n\sqrt{m}}\sum_{j=1}^n f'(\lambda_t^{j, n,W})\sqrt{\lambda_t^{j, n,W}}dB_t^{j}+ \langle\nu^{n,W}_t, \frac{\kappa^2(2-\beta_1\beta_2) \Phi}{m} f''+(\beta_1\kappa^2-2\gamma\Phi)f'\rangle dt \nonumber \\
     & \quad + \beta_1\beta_2\kappa^2 \left(\frac{n}{2m}\iint(x+y)\frac{f'(x)-f'(y)}{x-y}\nu^{n,W}_t(dx)\nu^{n,W}_t(dy)\right)dt \nonumber \\
     & =  dM_t^{(n,f)}+ \langle\nu^{n,W}_t, \frac{\kappa^2(2-\beta_1\beta_2) \Phi}{m} f''+(\beta_1\kappa^2-2\gamma\Phi)f'\rangle dt \nonumber\\
     &\quad + \beta_1\beta_2\kappa^2 \left(\frac{n}{2m}\iint(x+y)\frac{f'(x)-f'(y)}{x-y}\nu^{n,W}_t(dx)\nu^{n,W}_t(dy)\right)dt, \label{eq:integrodifcalculus}
\end{align}
with $\Phi: x\rightarrow x$ defined on $\mathbb{R}$ and $M^{(n,f)}$ a continuous martingale verifying

$$d\langle M^{(n,f)}\rangle _t= \frac{4\kappa^2 }{n^2m}\sum_{i=1}^n\mid \lambda_t^{i, n,W}f'(\lambda_t^{i, n,W})^2\mid dt.$$

 The reader will find in \cite{maecki2019universality} a proof of the tightness of the family $\{(\nu^{n,W}_t)_{t\geq 0}; n\geq 1\}$, which allows to conclude with the previous computations  that any accumulation point of this family satisfies the evolution equation (\ref{closedequ}). The reader will also find in  \cite{maecki2019universality} the proof of the uniqueness of the solution to the equation (\ref{closedequ}) in the case where $\nu_0$ admits a characteristic function and this function is analytic on a neighbourhood of the origin, this proof being based on the considerations made in Remark \ref{rmk:moments}.





Let us now prove the second part of the Theorem. Applying (\ref{closedequ}) with 
 $$f(v)=\frac{1}{z-v},$$
we get that $G_t(z)$ obeys
 
 \begin{eqnarray}
     G_t(z) &= & G_0(z) + \int_0^t \beta_1\kappa^2\frac{1}{(z-x)^2}-2\gamma\frac{x}{(z-x)^2}\nu_s(dx)ds +\frac{\alpha\beta_1\beta_2\kappa^2 }{2}\iint(x+y)\frac{f'(x)-f'(y)}{x-y}\nu_t(dx)\nu_t(dy). \nonumber 
\end{eqnarray}
We have
 \begin{align*}
     \frac{1}{2}\iint(x+y)\frac{f'(x)-f'(y)}{x-y}\nu_t(dx)\nu_t(dy)
     & =  \frac{1}{2}\iint\left(\frac{1}{(z-x)^2}-\frac{1}{(z-y)^2}\right)\frac{x+y}{x-y}\nu_t(dx)\nu_t(dy)  \\
     & =  \frac{1}{2}\iint \frac{(2z-x-y)(x+y)}{(z-x)^2(z-y)^2} \nu_t(dx)\nu_t(dy)  \\
     & =  \iint \frac{2zx-xy-x^2+z^2-z^2}{(z-x)^2(z-y)^2} \nu_t(dx)\nu_t(dy) \\
     & =  \iint -\frac{1}{(z-y)^2}-\frac{xy}{(z-x)^2(z-y)^2}+\frac{z^2}{(z-x)^2(z-y)^2} \nu_t(dx)\nu_t(dy)  \\
     & =  \iint -\frac{1}{(z-y)^2}+\frac{z^2}{(z-x)^2(z-y)^2} \nu_t(dx)\nu_t(dy)  -\left(\int\frac{x}{(z-x)^2}\nu_t(dx)\right)^2   \\
     & =  \iint -\frac{1}{(z-y)^2}+\frac{z^2}{(z-x)^2(z-y)^2} \nu_t(dx)\nu_t(dy) -\left(\int\frac{1}{x-z}+\frac{z}{(z-x)^2}\nu_t(dx)\right)^2   \\ 
     & =  -\int \frac{1}{(z-x)^2}\nu_t(dx)+z^2\left(\int\frac{1}{(z-x)^2} \nu_t(dx)\right)^2-\left(\int\frac{1}{z-x}\nu_t(dx)\right)^2\\
     &  +2z\left(\int\frac{1}{z-x}\nu_t(dx)\right)\left(\int \frac{1}{(z-x)^2}\nu_t(dx)\right)-z^2\left(\int\frac{1}{(z-x)^2}\nu_t(dx)\right)^2, 
 \end{align*}
   so that $G_t(z)$ obeys
 
 \begin{eqnarray}
     G_t(z) &= & G_0(z) \nonumber + \int_0^t \beta_1\kappa^2\frac{1}{(z-x)^2}-2\gamma\left(\frac{1}{x-z}+\frac{z}{(z-x)^2}\right)\nu_s(dx)ds \nonumber \\
     & & -\alpha\beta_1\beta_2\kappa^2 \left[\int \frac{1}{(z-x)^2}\nu_t(dx)+\left(\int\frac{1}{z-x}\nu_t(dx)\right)^2-2z\left(\int\frac{1}{z-x}\nu_t(dx)\right)\left(\int \frac{1}{(z-x)^2}\nu_t(dx)\right)\right].  \nonumber
 \end{eqnarray}
The conclusion is given replacing the previous terms by the corresponding derivatives of $G$.
\end{proof}

\begin{proof}[Proof of Proposition \ref{propoburgers}]
\begin{enumerate}[label=(\roman*)]

\item 
We first recall the Cauchy-Stieltjes transform of a Marcenko-Pastur law (see for instance \cite[Lemma 3.11]{baisilver} :
\[G_{\mu^{MP_{\rho,\sigma}}}(z)=\int\frac{\mu^{MP_{\rho,\sigma}}(dv)}{z-v}=\frac{-\sigma^2(1-\rho)+z-\sqrt{(z-\sigma^2-\rho\sigma^2)^2-4\rho\sigma^4}}{2\rho z\sigma^2},\] for all $z\in\mathbb{C}_+$.

We now want to find conditions on the functions $t\rightarrow\sigma(t)\in\mathbb{R}_+$ and $t\rightarrow\rho(t)\in\mathbb{R_+}$ with $\sigma(0)=\sigma_0$ and $\rho(0)=\rho_0$ such that $(t,z)\longrightarrow G_{\mu^{MP_{\rho(t),\sigma(t)}}}(z)$ is solution to the PDE (\ref{PDEBurgers}):
\begin{equation*}
    \begin{split}
        \frac{\partial}{\partial t}G_t(z) &=(\alpha\beta_1\beta_2\kappa^2 -\beta_1\kappa^2+2\gamma z)\frac{\partial}{\partial z}G_t(z) - 2\alpha\beta_1\beta_2\kappa^2  zG_t(z)\frac{\partial}{\partial z}G_t(z)-\alpha\beta_1\beta_2\kappa^2  G_t^2(z)+2\gamma G_t(z),\\
        G_0(z)&=\int\frac{\nu_0^W(dv)}{z-v}=\phi(z).
    \end{split}
\end{equation*}
We have
\begin{align}
    \frac{\partial}{\partial t}G_{\mu^{MP_{\rho(t),\sigma(t)}}}(z) & = -\frac{\sigma(z-\sigma^2)\dot\rho+2\dot\sigma\rho z}{2\rho^2z\sigma^3}-\frac{\sigma((1-\rho)\sigma^4-z(\rho+2)\sigma^2+z^2)\dot\rho+2(z-\sigma^2(1+\rho))z\dot\sigma\rho}{2\rho^2z\sigma^3\sqrt{(z-\sigma^2-\rho\sigma^2)^2-4\rho\sigma^4}}\nonumber \\
    \frac{\partial}{\partial z}G_{\mu^{MP_{\rho(t),\sigma(t)}}}(z) & =  \frac{1-\rho}{2\rho z^2}-\frac{z(1+\rho)-\sigma^2(1-\rho)^2}{2\rho z^2\sqrt{(z-\sigma^2-\rho\sigma^2)^2-4\rho\sigma^4}} \nonumber \\
    (G_{\mu^{MP_{\rho(t),\sigma(t)}}}(z))^2 & = \frac{\sigma^4(1-\rho)^2-2z\sigma^2+z^2+[z-\sigma^2(1-\rho)]\sqrt{(z-\sigma^2-\rho\sigma^2)^2-4\rho\sigma^4}}{2\rho^2z^2\sigma^4} \nonumber \\
    G_{\mu^{MP_{\rho(t),\sigma(t)}}}(z)\frac{\partial}{\partial z}G_{\mu^{MP_{\rho(t),\sigma(t)}}}(z) & =  \frac{-(1-\rho)^2\sigma^2+z}{2\rho^2z^3\sigma^2}+\frac{z^2+z\sigma^2[\rho^2\rho-2]+\sigma^4(1-\rho)^3}{2\rho^2z^3\sigma^2\sqrt{(z-\sigma^2-\rho\sigma^2)^2-4\rho\sigma^4}}.\nonumber
\end{align}

The equation solved by $G$ can be written for $z\in\mathbb{C}_+, t\geq0$:

\begin{equation}
\label{eq:H0}
    H(z,t)=0
\end{equation}
with
\begin{equation*}
    \begin{split}
        H(z,t)&=\frac{z\sigma^2(z-\sigma^2)\dot\rho+2\sigma\rho\dot\sigma z^2+\beta_1\kappa^2(1-\rho)(\beta_2\alpha -\rho)\sigma^4+2\sigma^2\rho\gamma z^2-\beta_1\beta_2\alpha \kappa^2  z^2}{2\rho^2\sigma^4z^2}\\
        &-\frac{z((1-\rho)\sigma^4-z(\rho+2)\sigma^2+z^2)\sigma^2\dot\rho+2z^2(z-(1+\rho)\sigma^2)\rho\sigma\dot\sigma+\beta_1\kappa^2(\rho-1)^2(\beta_2\alpha -\rho)\sigma^6
        }{2\rho^2\sigma^4z^2\sqrt{(z-\sigma^2-\rho\sigma^2)^2-4\rho\sigma^4}}\\
        &+\frac{
        -z((2\gamma z+\beta_1\kappa^2)\rho-\beta_1\beta_2\alpha \kappa^2 )(\rho+1)\sigma^4+z^2((2\gamma z+\beta_1\beta_2\alpha \kappa^2 )\rho+\beta_1\beta_2\alpha \kappa^2 )\sigma^2-\beta_1\beta_2\alpha \kappa^2  z^3}{2\rho^2\sigma^4z^2\sqrt{(z-\sigma^2-\rho\sigma^2)^2-4\rho\sigma^4}}.
    \end{split}
\end{equation*}
Using the fact that 
\begin{equation*}
    ((z-\sigma^2-\rho\sigma^2)^2-4\rho\sigma^4)^{-\frac{1}{2}}=\frac{1}{z}+\frac{\sigma^2(1+\rho)}{z^2}+\sigma^4\frac{\rho^2+4\rho+1}{z^3}+o_{|z|\rightarrow+\infty}\left(\frac{1}{|z|^3}\right),
\end{equation*}
we have for $z\in\mathbb{C}_+$, $t\geq0$ :
\begin{equation*}
\begin{split}
        H(z,t)=&\frac{\sigma^2\dot{\rho}+2\sigma\rho\dot{\sigma}-\beta_1\beta_2\alpha \kappa^2 +2\gamma\rho\sigma^2}{\sigma^4\rho^2}-\frac{\dot\rho}{\rho^2}\frac{1}{z}\\
        &-\frac{\beta_1\kappa^2(\rho-1)(\beta_2\alpha-\rho)-4\sigma\rho^2\dot\sigma+(\beta_1\kappa^2-4\gamma\sigma^2)\rho^2+\beta_1\kappa^2(1+\beta_2\alpha )\rho-\beta_1\beta_2\alpha\kappa^2 }{2\rho^2}\frac{1}{z^2}+o_{|z|\rightarrow+\infty}\left(\frac{1}{|z|^2}\right),
\end{split}
\end{equation*}

which gives in particular by identity (\ref{eq:H0}) : 
\begin{equation*}
    \begin{split}
        \sigma^2\dot{\rho}+2\sigma\rho\dot{\sigma}-\beta_1\beta_2\alpha \kappa^2 +2\gamma\rho\sigma^2 &=0,\\
        \dot\rho &=0,\\
        \beta_1\kappa^2(\rho-1)(\beta_2\alpha-\rho)-4\sigma\rho^2\dot\sigma+(\beta_1\kappa^2-4\gamma\sigma^2)\rho^2+\beta_1\kappa^2(1+\beta_2\alpha )\rho-\beta_1\beta_2\alpha\kappa^2 &=0.
    \end{split}
\end{equation*}

We thus have
\[\rho(t)=\beta_2\alpha  ,\]
\[\dot{\sigma^2}=\beta_1\kappa^2-2\gamma\sigma^2,\]
and 
\begin{equation*}
        \sigma^2(t)= \left\{
        \begin{aligned}
            &\left(\sigma^2_0-\frac{\beta_1\kappa^2}{2\gamma}\right)\mathrm{e}^{-2\gamma t}+\frac{\beta_1\kappa^2}{2\gamma} \quad\text{if}\quad \gamma\neq0\\
            &\sigma_0^2+\beta_1\kappa^2t \quad\text{if}\quad \gamma=0,
        \end{aligned}
        \right.
    \end{equation*} for all $t\in\mathbb{R}_+$.

Reciprocally, we verify with this definition of $\sigma$ that $G_\mu^{MP_{\beta_2\alpha,\sigma(t)}}$ is a solution to the PDE (\ref{PDEBurgers}).
\item Let us consider the stationary version of the PDE (\ref{PDEBurgers}) :
\begin{equation*}
    \begin{split}
        (\beta_1\beta_2\alpha\kappa^2 -\beta_1\kappa^2+2\gamma z)\frac{\mathrm{d}}{\mathrm{d z}}G(z) - 2\beta_1\beta_2\alpha\kappa^2  zG(z)\frac{\mathrm{d}}{\mathrm{d z}}G(z)-\beta_1\beta_2\alpha\kappa^2  G^2(z)+2\gamma G(z)=0.
    \end{split}
\end{equation*}
We can integrate it 
\begin{equation*}
    \begin{split}
         -\beta_1\beta_2\alpha\kappa^2  zG^2(z)+(2\gamma z+(\beta_2\alpha-1)\beta_1\kappa^2) G(z)=C,
    \end{split}
\end{equation*}
where $C\in\mathbb{C}$ is an integration constant. It gives 
\begin{equation*}
    G(z) = \frac{2\gamma z+(\beta_2\alpha-1)\beta_1\kappa^2\pm\sqrt{(2\gamma z+(\beta_2\alpha-1)\beta_1\kappa^2)^2-4C\beta_1\beta_2\alpha\kappa^2z}}{2\beta_1\beta_2\alpha\kappa^2 z}.
\end{equation*}
For this function to be the Cauchy-Stieltjes transform of a probability measure on $\mathbb{R}$, we need it to verify the condition given in assertion $(3)$ of Theorem \ref{thmcauchytransform} : \[\underset{y\rightarrow\infty}{\limsup}y|G(\mathrm{i}y)|=1,\] necessarily, the sign $\pm$ must be replaced by a minus sign and $C=2\gamma$. Thus, the only stationary solution to the PDE (\ref{PDEBurgers}) corresponding to the Cauchy-Stieltjes transform of a probability measure on $\mathbb{R}$ is
\begin{equation*}
\begin{split}
    G(z) &= \frac{2\gamma z+(\beta_2\alpha-1)\beta_1\kappa^2-\sqrt{(2\gamma z+(\beta_2\alpha-1)\beta_1\kappa^2)^2-8\gamma\beta_1\beta_2\alpha\kappa^2z}}{2\beta_1\beta_2\alpha\kappa^2 z}\\
    &= \frac{-\sigma^2(1-\rho)+z-\sqrt{(z-\sigma^2-\rho\sigma^2)^2-4\rho\sigma^4}}{2\rho z\sigma^2},
\end{split}
\end{equation*}
with $\rho = \beta_2\alpha$ and $\sigma^2=\frac{\beta_1\kappa^2}{2\gamma}$. We recognize the Cauchy-Stieltjes transform of the Marcenko-Pastur distribution with parameters $\rho$ and $\sigma$, and verify that $G$ is solution to the stationary version of the PDE (\ref{PDEBurgers}). Under the assumptions of Theorem \ref{th:burgers}, and applying Theorem \ref{th:burgers} with  $\nu_0=\mu^{MP_{\beta_2\alpha,\sigma_0}}$, then $\mu^{MP_{\beta_2\alpha,\sigma_0}}$ is the unique solution to the equation $(\ref{closedequ})$.

\end{enumerate}

\end{proof}

\begin{proof}[Proof of Theorem \ref{th:betamfW}]
By  Theorem \ref{th:burgers},  the  integro-differential equation (\ref{closedequ}) is  verified by any accumulation point of the family $\{(\nu^{n,W}_t)_{t\geq0},n\geq1\}$.

Let $n\leq m$ with $\frac{n}{m}\underset{n\rightarrow+\infty}{\longrightarrow}\beta_2\alpha$. Let us define the  stochastic process  $(M_t)_t$ taking its values in the space of  $n\times m$ matrices with real entries verifying the following SDE
\begin{equation*}
    dM_t = \sqrt{\beta_1}\kappa dW_t-\gamma M_tdt.
\end{equation*}
The eigenvalues of $\frac{1}{m}M_tM_t^*$ verify the SDE (\ref{eds}) with $\kappa$ replaced by $\sqrt{\beta_1}\kappa$. By Theorem \ref{th:burgers} (see more precisely the computation (\ref{eq:integrodifcalculus}) in the proof), the integro-differential equation (\ref{closedequ}) is also verified by any accumulation point of the family of the  empirical spectral measures of $\left(\frac{1}{m}M_tM_t^*\right)_{1\leq n\leq m}$ with $\frac{n}{m}\underset{n\rightarrow+\infty}{\longrightarrow}\beta_2\alpha$, which ends the proof.

By Theorem \ref{th:burgers}, we have  uniqueness of the solution to the integro-differential equation (\ref{closedequ}), and by Theorem \ref{th:mfl} we have the expression of the limit of $\mu^{n,W}_t=\sqrt{\nu^{n,W}_t}$ for all $t\geq0$. 
\end{proof}

\begin{proof}[Proof of Theorem \ref{th:limcommutW}]
For all $n\in \mathbb{N}^*$, an application of \cite[Lemma 3.1 and Proposition 2.8]{JK} shows that in convergence in law
\[\nu^{n,W}_t\underset{t\rightarrow+\infty}{\longrightarrow}\nu^{n,W},\] 
where $\nu^{n,W}$ is defined such as in equation (\ref{prop:stat}).
By equation (\ref{prop:stat}), we have
\[\nu^{n,W}\underset{\underset{n/m\rightarrow\alpha}{n\rightarrow+\infty}}{\longrightarrow}\mu^{MP_{\alpha,\sigma_\infty}}.\] 
Theorem \ref{th:betamfW} gives the two other limits and concludes the proof.
\end{proof}

\paragraph{Acknowledgement} : I thank Benjamin Jourdain and Djalil Chafaï for numerous fruitful discussions.

\nocite{*}
\bibliographystyle{amsalpha}
\bibliography{biblio}
\end{document}